\documentclass[12pt]{amsart}
\usepackage{amssymb,amscd}
\usepackage{verbatim}

\let\frak\mathfrak

\def\>{\relax\ifmmode\mskip.666667\thinmuskip\relax\else\kern.111111em\fi}
\def\<{\relax\ifmmode\mskip-.333333\thinmuskip\relax\else\kern-.0555556em\fi}
\def\vsk#1>{\vskip#1\baselineskip}
\def\vv#1>{\vadjust{\vsk#1>}\ignorespaces}
\def\vvn#1>{\vadjust{\nobreak\vsk#1>\nobreak}\ignorespaces}
 \let\alb\allowbreak
\def\fratop{\genfrac{}{}{0pt}1}
\def\satop#1#2{\fratop{\scriptstyle#1}{\scriptstyle#2}}
  \let\ssize\scriptstyle
\let\sssize\scriptscriptstyle

\let\Medskip\medskip
\def\medskip{\par\Medskip}
\let\Bigskip\bigskip
\def\bigskip{\par\Bigskip}

\let\Maketitle\maketitle
\def\maketitle{\Maketitle\thispagestyle{empty}\let\maketitle\empty}

\newtheorem{thm}{Theorem}[section]
\newtheorem{cor}[thm]{Corollary}
\newtheorem{lem}[thm]{Lemma}
\newtheorem{prop}[thm]{Proposition}

\numberwithin{equation}{section}

\theoremstyle{definition}

\newtheorem*{example}{Example}

\let\mc\mathcal
\let\nc\newcommand

\let\dl\delta
\let\Dl\Delta
\let\eps\varepsilon

\let\la\lambda

\let\pho\phi
\let\phi\varphi
\let\si\sigma

\let\Ups\Upsilon

\let\der\partial
\let\Hat\widehat

\let\ox\otimes
\let\Tilde\widetilde
\let\bra\langle
\let\ket\rangle

\let\ge\geqslant
\let\geq\geqslant
\let\le\leqslant
\let\leq\leqslant

\let\on\operatorname
\let\bi\bibitem
\let\bs\boldsymbol

\def\C{{\mathbb C}}
\def\Z{{\mathbb Z}}

\def\F{{\mc F}}

\def\+#1{^{\{#1\}}}
\def\lsym#1{#1\alb\dots\relax#1\alb}
\def\lc{\lsym,}

\def\End{\on{End}}

\def\Hom{\on{Hom}}

\def\ii{i,\<\>i}
\def\ij{i,\<\>j}
\def\ik{i,\<\>k}
\def\il{i,\<\>l}
\def\ji{j,\<\>i}
\def\jj{j,\<\>j}
\def\jk{j,\<\>k}
\def\kj{k,\<\>j}
\def\kl{k,\<\>l}

\def\ioi{i+1,\<\>i}

\def\ppo{p,\<\>p+1}
\def\pop{p+1,\<\>p}
\def\pci{p,\<\>i}
\def\pcj{p,\<\>j}
\def\poi{p+1,\<\>i}
\def\poj{p+1,\<\>j}

\def\gln{\mathfrak{gl}_N}

\def\Ugln{U(\gln)}

\def\Yn{Y\<(\gln)}

\def\beq{\begin{equation}}
\def\eeq{\end{equation}}
\def\be{\begin{equation*}}
\def\ee{\end{equation*}}

\nc{\bea}{\begin{eqnarray*}}
\nc{\eea}{\end{eqnarray*}}
\nc{\bean}{\begin{eqnarray}}
\nc{\eean}{\end{eqnarray}}
\nc{\Ref}[1]{{\rm(\ref{#1})}}

\def\n{{\mathfrak n}}

\let\ga\gamma
\let\Ga\Gamma

\nc{\Il}{{\mc I_{\bs\la}}}
\nc{\bla}{{\bs\la}}
\nc{\Fla}{\F_\bla}
\nc{\tfl}{{T^*\Fla}}
\nc{\GL}{{GL_n(\C)}}
\nc{\GLC}{{GL_n(\C)\times\C^*}}

\let\aal\al 

\def\zzz{z_1\lc z_n}

\def\Czh{\C[\zzz,h]}

\def\ty{\Tilde Y\<(\gln)}

\let\sd s 

\def\ib{\bs i}
\def\jb{\bs j}

\def\iib{\ib,\<\>\ib}
\def\ijb{\ib,\<\>\jb}

\def\jib{\jb,\<\>\ib}

\def\zb{\bs z}

\def\ip{\<\>i\>\prime}
\def\ipi{\>\prime\<\>i}

\def\Xin{X^\infty}

\def\Xk{X^\kk}

\def\ddk_#1{\kk_{#1}\<\>\frac\der{\der\<\>\kk_{#1}}}

\def\bul{\mathbin{\raise.2ex\hbox{$\sssize\bullet$}}}
\def\intt{\mathchoice
{\mathop{\raise.2ex\rlap{$\,\,\ssize\backslash$}{\intop}}\nolimits}
{\mathop{\raise.3ex\rlap{$\,\sssize\backslash$}{\intop}}\nolimits}
{\mathop{\raise.1ex\rlap{$\sssize\>\backslash$}{\intop}}\nolimits}
{\mathop{\rlap{$\sssize\<\>\backslash$}{\intop}}\nolimits}}

\let\gak\gamma 
\let\kk q 
\let\kp\kappa 
\let\cc c
\def\kkk{\kk_1\lc\kk_N}

\let\Ko K
\def\Kh{\Hat\Ko}

\def\GZ/{Gelfand-Zetlin}
\def\KZ/{{\slshape KZ\/}}
\def\qKZ/{{\slshape qKZ\/}}
\def\XXX/{{\slshape XXX\/}}

\def\zz{{\bs z}}
\def\qq{{\bs q}}
\def\TT{{\bs t}}

\def\Sym{\on{Sym}}

\def\ss{{\bs s}}

\nc{\A}{{\mc C}}
\def\glnn{{\frak{gl}_n}}

\def\GG{{\bs \Ga}}

\def\XX{{\mc X}}

\def\St{{\on{Stab}}}
\def\xx{{\bs x}}
\def\Czh{{(\C^N)^{\otimes n}\otimes\C(\zz;h)}}

\let\8\infty
\gdef\){\>\]}
\gdef\){\RIfM@\mskip.333333\thinmuskip\relax\else\kern.0555556em\fi}
\gdef\]{{\!\!\;}}
\def\fd/{fin\-ite-dim\-en\-sion\-al}
\def\glk{{$\frak{gl}_k$}}
\def\gkmod/{\$\glk$-module}
\def\gnmod/{\$\gln$-module}

\begin{document}

\hrule width0pt
\vsk->

\title[Hypergeometric solutions of the quantum differential
	equation ]
{Hypergeometric solutions of the quantum differential
	equation of the cotangent bundle\\ of a partial flag variety}

\author
[V.\,Tarasov and A.\,Varchenko]
{V.\,Tarasov$\>^\circ$ and A.\,Varchenko$\>^\diamond$}

\maketitle

\begin{center}
{\it $\kern-.4em^\circ\<$Department of Mathematical Sciences,
Indiana University\,--\>Purdue University Indianapolis\kern-.4em\\
402 North Blackford St, Indianapolis, IN 46202-3216, USA\/}

\vsk.5>
{\it $^\circ\<$St.\,Petersburg Branch of Steklov Mathematical Institute\\
Fontanka 27, St.\,Petersburg, 191023, Russia\/}

\vsk.5>
{\it $^{\diamond}\<$Department of Mathematics, University
of North Carolina at Chapel Hill\\ Chapel Hill, NC 27599-3250, USA\/}

\end{center}

{\let\thefootnote\relax
\footnotetext{\vsk-.8>\noindent
$^\circ\<${\sl E\>-mail}:\enspace vt@math.iupui.edu\>, vt@pdmi.ras.ru\>,
supported in part by NSF grant DMS-0901616\\
$^\diamond\<${\sl E\>-mail}:\enspace anv@email.unc.edu\>,
supported in part by NSF grant DMS-1101508}}

\begin{abstract}
We describe hypergeometric solutions of the quantum differential
equation of the cotangent bundle of a $\frak{gl}_n$ partial flag variety.
These hypergeometric solutions manifest the Landau-Ginzburg mirror
symmetry for the cotangent bundle of a partial flag variety.
\end{abstract}

\bigskip

{\small\tableofcontents}

\setcounter{footnote}{0}
\renewcommand{\thefootnote}{\arabic{footnote}}

\section{Introduction}

In \cite{MO}, D.\,Maulik and A.\,Okounkov develop a general theory connecting quantum groups
and equivariant quantum cohomology of Nakajima quiver varieties, see [N1, N2]. In particular, in \cite{MO}
the operators of quantum multiplication by divisors
are described. As well-known, these operators determine the equivariant
quantum differential equation of a quiver variety.
In this paper we apply this description to the cotangent bundles of $\frak{gl}_n$ partial flag varieties
and construct hypergeometric solutions of the associated equivariant quantum differential equation.

In \cite{GRTV} and \cite{RTV}, the equivariant quantum differential equation of the cotangent bundle of a $\frak{gl}_n$
partial flag variety was identified with the trigonometric dynamical differential equation introduced in
\cite{TV4}.
By the $(\gln,\frak{gl}_n)$-duality of \cite{TV4}, the trigonometric dynamical differential equation
is identified with the trigonometric KZ differential equation. Hypergeometric solutions
of the trigonometric KZ differential equation were constructed in \cite{SV, MV}.
By using this sequence of isomorphisms we obtain
hypergeometric solutions of the equivariant quantum differential equation.
The hypergeometric solutions have the form
\bean
\label{solhi}
I_\psi(\zz; \tilde \qq; h; \kappa) = \int_{\psi( \zz;\tilde \qq;h;\kappa)} \Phi(\ss;\zz;\tilde\qq; h)^{h/\kappa} \omega(\ss;\zz;\tilde\qq;h) d\bs s,
\eean
where $\zz=(z_1,\dots,z_n)$, $h$ are equivariant parameters, $\tilde\qq=(\tilde q_1,\dots,\tilde q_N)$ quantum parameters,
$\bs s=(s^{(i)}_j)$ integration variables, $\kappa$ the parameter of the
differential equation, $\psi(\zz;\tilde \qq;h;\kappa)$ the integration cycle in the $\bs s$-space,
$\omega(\ss;\zz;\tilde \qq;h)$ a rational function, $\Phi(\ss;\zz;\tilde \qq;h)$ the master function, see Corollary \ref{cor main}.

Studying solutions of the quantum differential equation may lead to better understanding
Gromow-Witten invariants of the cotangent bundle of a partial flag variety, c.f. Givental's study
of the $J$-function in \cite{G1, G2, G3}.

The existence of solutions of the quantum differential equation
as such oscillatory integrals manifests the Landau-Ginzburg mirror
symmetry for the cotangent bundle of a partial flag variety. One may think that the logarithm of the master
function is the Landau-Ginzburg potential of the mirror dual object. In particular,
one may expect that the algebra of functions on the critical set of the master function
$\Phi(\,\cdot\,; \zz;\tilde \qq;h)$ is isomorphic to the corresponding localization of the equivariant
quantum cohomology algebra of the cotangent bundle of a partial flag variety, see Section \ref{sec exa}
and similar statements for Bethe algebras in
\cite{MTV6, MTV8}.

\smallskip

The trigonometric KZ differential equation and trigonometric dynamical differential equation come in pairs
with compatible difference equations. The trigonometric KZ differential equation is compatible
with the rational dynamical difference equation introduced in \cite{TV3}.
The trigonometric dynamical differential equation is compatible
with the rational \qKZ/ difference equation, as shown in \cite{TV4}. Under the $(\gln,\frak{gl}_n)$-duality,
the pair consisting of the trigonometric KZ differential equation and rational dynamical
difference equation is identified with the pair consisting of
the trigonometric dynamical differential equation and rational \qKZ/ difference equation, see \cite{TV4}.
Under the identification of the trigonometric dynamical differential equation
with the equivariant quantum differential equation of \cite{BMO}, the rational \qKZ/ difference operators
are identified with the shift operators of \cite{MO}, see also \cite{BMO}.

It was shown in \cite{MV}, that the hypergeometric solutions for the trigonometric KZ differential equation
also satisfy the difference dynamical equation. By using all of the above identifications, we conclude
that the hypergeometric solutions \Ref{solhi} of the equivariant quantum differential equation give also
flat sections of the difference connection
defined by the shift operators of \cite{MO} on the equivariant quantum cohomology of the cotangent bundle of a partial $\frak{gl}_n$ flag variety.

\smallskip
As shown in \cite{TV1, TV2}, the \qKZ/ equation has solutions in the form of multidimensional
$q$-hypergeometric integrals. These $q$-hypergeometric integrals are expected to satisfy the compatible
dynamical differential equation. That will imply that the quantum differential equation of \cite{MO}
has solutions in the form of multidimensional $q$-hypergeometric
integrals. See an example of such solutions in Section 8.4 of \cite{GRTV}.
We plan to develop these $q$-hypergeometric solutions in a separate paper.

\medskip

\section{Cotangent bundles of partial flag varieties}
\label{sec EQ}

\subsection{Partial flag varieties}
\label{sec Partial flag varieties}
Fix natural numbers $N, n$. Let \,$\bla\in\Z^N_{\geq 0}$, \,$|\bla|=\la_1+\dots+\la_N =n$.
Consider the partial flag variety
\;$\Fla$ parametrizing chains of subspaces
\be
0\,=\,F_0\subset F_1\lsym\subset F_N =\,\C^n
\ee
with \;$\dim F_i/F_{i-1}=\la_i$, \;$i=1\lc N$.
Denote by \,$\tfl$ the cotangent bundle of \;$\Fla$.
Denote
\bea
\XX_n = \cup_{|\bla|=n} T^*\F_\bla .
\eea
\begin{example}
If $n=1$, then $\bla=(0,\dots,0,1_i,0,\dots,0)$, $T^*\F_\bla$ is a point and $\XX_1$ is the union of $N$ points.

If $n=2$ then $\bla= (0,\dots,0,1_i,0, \dots,0,1_j,0,\dots,0)$ or $\bla= (0,\dots,0,2_i,0,\dots,0)$.
In the first case $T^*\F_\bla$ is the cotangent bundle of projective line, in the second case $T^*\F_\bla$ is a point.
Thus $\XX_2$ is the union of $N$ points and $N(N-1)/2$ copies of the cotangent bundle of projective line.

\end{example}

Let $I=(I_1\lc I_N)$ be a partition of $\{1\lc n\}$ into disjoint subsets
$I_1\lc I_N$. Denote $\Il$ the set of all partitions $I$ with
$|I_j|=\la_j$, \;$j=1,\dots N$. Denote $\mc I_n = \cup_{|\bla|=n}\Il$.

Let $u_1,\dots,u_n$ be the standard basis of $\C^n$.
For any $I\in\Il$, let $x_I\in \F_\bla$ be the point corresponding to the coordinate flag
$F_1\subset\dots\subset F_N$, where $F_i$
\,is the span of the standard basis vectors \;$u_j\in\C^n$ with
\,$j\in I_1\lsym\cup I_i$. We embed $\F_\bla$ in $T^*\F_\bla$ as the zero section and consider
the points $x_I$ as points of $T^*\F_\bla$.

\subsection{Equivariant cohomology}
\label{sec:equiv}

Denote $G=GL_n(\C)\times \C^\times$.
Let \,$A\!\subset GL_n(\C)$ \,be the torus of diagonal matrices.
Denote $T=A\times\C^\times$ the subgroup of $G$.

The groups \,$A\!\subset GL_n$ act on \;$\C^n\<$ and hence on \,$\tfl$.
Let the group \;$\C^\times\<$ act on \,$\tfl$ by multiplication in each fiber.
We denote by $-h$ its $\C^\times$-weight.

\vsk.2>
We consider the equivariant cohomology algebras $H^*_{T}(\tfl;\C)$ and
\bea
H_T^*(\XX_n)=
\oplus_{|\bla|=n} H^*_{T}(\tfl;\C).
\eea
Denote by $\Ga_i=\{\ga_{i,1}\lc\ga_{i,\la_i}\}$ the set of
the Chern roots of the bundle over $\Fla$ with fiber $F_i/F_{i-1}$.
Let \;$\GG=(\Ga_1\<\>\lsym;\Ga_N)$. Denote by $\zb=\{\zzz\}$ the Chern roots
corresponding to the factors of the torus $T$.
Then
\vvn-.2>
\bea
\label{Hrel}
H^*_{T}(\tfl)
\, = \,\C[\GG]^{S_{\la_1}\times\dots\times S_{\la_N}}
\otimes \C[\zz]\otimes\C[h]
\>\Big/\Bigl\bra\,
\prod_{i=1}^N\prod_{j=1}^{\la_i}\,(u-\ga_{\ij})\,=\,\prod_{a=1}^n\,(u-z_a)
\Bigr\ket\,.
\eea
The cohomology
$H^*_{T}(\tfl)$ is a module over $H^*_{T}({pt};\C)=\C[\zz]\otimes\C[h]$.

\begin{example}
If $n=1$, then
\bea
H_T^*(\XX_1) = \oplus_{i=1}^N H_T^*(T^*\F_{(0,\dots,0,1_i,0,\dots,0)})
\eea
is naturally isomorphic to $\C^N\otimes\C[z_1;h]$ with basis
$v_i=(0,\dots,0,1_i,0,\dots,0)$, $i=1,\dots,N$.
\end{example}

For $i=1,\dots,N$, denote $\la^{(i)}=\la_1+\dots+\la_i$. Denote $\Theta_i=\{\theta_{i,1}, \dots, \theta_{i,\la^{(i)}}\}$
the Chern roots of the bundle $\bs F_i$
over $\F_\bla$ with fiber $F_i$. Let $\bs \Theta=(\Theta_1,\dots,\Theta_{N})$. The relations
\bea
\prod_{j=1}^{\la^{(i)}}(u-\theta_{i,j}) = \prod_{\ell=1}^i\prod_{j=1}^{\la_i}\,(u-\ga_{\ij}), \qquad i=1,\dots,N,
\eea
define the homomorphism
\bea
\C[\bs \Theta]^{S_{\la^{(1)}}\times\dots\times S_{\la^{(N)}}}
\otimes \C[\zz]\otimes\C[h]\to H^*_{T}(\tfl).
\eea

\section{Yangian }
\label{sec Yang}

\subsection{Yangian $\Yn$}
\label{sec yangian}

The Yangian $\Yn$ is the unital associative algebra with generators
\,$T_{\ij}\+s$ \,for \,$i,j=1\lc N$, \;$s\in\Z_{>0}$, \,subject to relations
\vvn.3>
\bea
\label{ijkl}
(u-v)\>\bigl[\<\>T_{\ij}(u)\>,T_{\kl}(v)\<\>\bigr]\>=\,
T_{\kj}(v)\>T_{\il}(u)-T_{\kj}(u)\>T_{\il}(v)\,,\qquad i,j,k,l=1\lc N\,,
\vv.3>
\eea
where
$ T_{\ij}(u)=\dl_{\ij}+\sum_{s=1}^\infty\,T_{\ij}\+s\>u^{-s}\>.$
The Yangian $\Yn$ is a Hopf algebra with the coproduct
\;$\Dl:\Yn\to\Yn\ox\Yn$ \,given by
\vvn.16>
\;$\Dl\bigl(T_{\ij}(u)\bigr)=\sum_{k=1}^N\,T_{\kj}(u)\ox T_{\ik}(u)$ \,for
\,$i,j=1\lc N$\>. The Yangian $\Yn$ contains, as a Hopf subalgebra, the universal enveloping algebra \>$\Ugln$ of the Lie algebra $\gln$.
The embedding is given by \,$e_{\ij}\mapsto T_{\ji}\+1$, where $e_{\ij}$ are standard standard generators of \,$\gln$.

\vsk.2>
Notice that \,$\bigl[\<\>T_{\ij}\+1,T_{\kl}\+s\<\>\bigr]\>=\,
\dl_{\il}\>T_{\kj}\+s-\dl_{\jk}\>T_{\il}\+s$ \,for \,$i,j,k,l=1\lc N$,
\;$s\in\Z_{>0}$\,, \,which implies that the Yangian \>$\Yn$ is generated by
the elements \,$T_{\ii+1}\+1\>,\>T_{\ioi}\+1$, \;$i=1\lc N-1$, \,and
\,$T_{1,1}\+s$, \;$s>0$.

\subsection{Algebra \,$\ty\>$}
\label{sec tilde Y}

In this section we follow \cite[Section 3.3]{GRTV}.
In formulas of that Section 3.3 we replace $h$ with $-h$.

Let \,$\ty$ be the subalgebra of \,$\Yn\ox\C[h]$ generated over \,$\C\>$
by \,$\C[h]$ and the elements \,$(-h)^{s-1}\>T_{\ij}\+s$ \,for \,$i,j=1\lc N$,
\;$s>0$. Equivalently, the subalgebra \,$\ty$ \,is generated over \,$\C\>$ by
\,$\C[h]$ and the elements \,$T_{\ii+1}\+1\>,\>T_{\ioi}\+1$,
\;$i=1\lc N-1$, \,and \,$(-h)^{s-1}\>T_{1,1}\+s$, \;$s>0$.

For \;$p=1\lc N$, \;$\ib=\{1\leq i_1<\dots<i_p\leq N\}$,
\;$\jb=\{1\leq j_1<\dots<j_p\leq N\}$, define
\vvn-.3>
\be
M_{\ijb}(u) =\sum_{\si\in S_p}(-1)^\si\,
T_{i_1,j_{\si(1)}}(u)\dots T_{i_p,j_{\si(p)}}(u-p+1)\,.
\ee
Introduce the series \,$A_1(u)\lc A_N(u)$, \,$E_1(u)\lc E_{N-1}(u)$,
\,$F_1(u)\lc F_{N-1}(u)$:
\vvn.1>
\begin{gather}
\label{A}
A_p(u)\,=\,M_{\iib}(-u/h)\,=\,
1+\sum_{s=1}^\infty\,(-h)^s\>A_{p,s}\,u^{-s}\,,
\\[3pt]
E_p(u)\,=\,-h^{-1}M_{\jib}(-u/h)\>\bigl(M_{\iib}(-u/h)\bigr)^{-1}\,=\,
\sum_{s=1}^\infty\,(-h)^{s-1}\>E_{p,s}\,u^{-s}\,,
\label{EF}
\\[3pt]
F_p(u)\,=\,-h^{-1}\bigl(M_{\iib}(-u/h)\bigr)^{-1}M_{\ijb}(-u/h)\,=\,
\sum_{s=1}^\infty\,(-h)^{s-1}\>F_{p,s}\,u^{-s}\,,
\notag
\\[-20pt]
\notag
\end{gather}
where in formulas \Ref{A} and \Ref{EF} we have \,$\ib=\{1\lc p\>\}$\>, \,$\jb=\{1\lc p-1,p+1\>\}$\>.
Observe that \,$E_{p,1}=T_{\pop}\+1$\>, \,$F_{p,1}=T_{\ppo}\+1$
\,and \,$A_{1,s}=T_{1,1}\+s$, so the coefficients of the series
\,$E_p(u)$, \,$F_p(u)$ and \,$h^{-1}(A_p(u)-1\bigr)$ together with \,$\C[h]$
generate \,$\ty$.
In what follows we will describe actions of the algebra \,$\ty$ by using series
\Ref{A}, \Ref{EF}.

\subsection{$\ty$-action on $(\C^N)^{\otimes n}\otimes\C[\zz;h]$}
\label{sec yang act C^N}

Set
\bea
L(u)\,=\,(u-z_n-h\<\>P^{(0,n)})\dots(u-z_1-h\<\>P^{(0,1)})\,,
\eea
where the factors of \,$\C^N\!\ox (\C^N)^{\otimes n}$ are labeled by \,$0,1\lc n$ and $P^{(i,j)}$ is the permutation of the
$i$-th and $j$-th factors.
The operator
$L(u)$ is a polynomial in \,$u,\zz,h$ \,with values
in \,$\End(\C^N\!\ox (\C^N)^{\otimes n})$. We consider $L(u)$ as an $N\!\times\!N$ matrix
with \,$\End(V)\ox\C[u;\zz; h]\>$-valued entries \,$L_{\ij}(u)$.

\begin{prop} [Proposition 4.1 in \cite{GRTV}]

The assignment
\vvn-.2>
\beq
\label{pho}
\pho\bigl(T_{\ij}(-u/h)\bigr)\,=\,
L_{\ij}(u)\,\prod_{a=1}^n\,(u-z_a)^{-1}
\vv-.1>
\eeq
defines the action of the algebra \,$\ty$
\vvn.2>
on \,$(\C^N)^{\otimes n}\otimes\C[\zz;h]$\,. Here the right-hand side of \Ref{pho} is a series
in \,$u^{-1}$ with coefficients in \,$\End((\C^N)^{\otimes n})\otimes\C[\zz;h])$\>.
\end{prop}

Under this action, the subalgebra $U(\gln)\subset\ty$ acts on $(\C^N)^{\otimes n}\otimes\C[\zz;h]$
in the standard way: any element \,$x\in\gln$ \,acts as $x^{(1)}\lsym+x^{(n)}$.
The action $\pho$ was denoted in \cite{GRTV} by $\pho^+$.

\subsection{$\ty$-action on $H^*_T(\XX_n)$ according to \cite{RTV}}
\label{sec cohom and Yang}

We define the \>$\ty$-action $\rho$ on \;$H_T^*(\XX_n)$ by formulas \Ref{Arho}, \Ref{rho E}, \Ref{rho F} below.
We define
\;$\rho\bigl(A_p(u)\bigr): H^*_T(T^*\F_\bla)\to H^*_T(T^*\F_\bla)$ by
\vvn.2>
\beq
\label{Arho}
\rho\bigl(A_p(u)\bigr)\>:\>[f]\;\mapsto\,
\Bigl[\>f(\GG;\zb;h)\;\prod_{a=1}^p\,\prod_{i=1}^{\la_p}
\;\Bigl(1-\frac h{u-\ga_{\pci}}\Bigr)\Bigr]\,,
\vv.3>
\eeq
for \,$p=1\lc N$. In particular,
\vvn.4>
\bean
\label{rhoX inf}
\rho(\Xin_i)\>:\>[f]\;\mapsto\,
[\>(\ga_{i,1}\lsym+\<\ga_{i,\la_i})\>f(\GG;\zb;h)]\,,\qquad i=1\lc N\,.
\kern-4em
\eean
Let \;$\aal_1\lc\aal_{N-1}$ \,be simple roots,
\,$\aal_p=(0\lc0,1,\<-\>1,0\lc0)$, with \,$p-1$ first zeros.
We define
\vvn-.6>
\bea
\rho\bigl(E_p(u)\bigr)\<\>:\>H^*_T(T^*\F_{\bla\<\>-\aal_p})\,\mapsto\,H^*_T(T^*\F_\bla)\,,
\eea
\bean
\label{rho E}
\phantom{aaaaa}
\rho\bigl(E_p(u)\bigr)\<\>:\>[f]\;\mapsto\,\biggl[\;
\sum_{i=1}^{\la_p}\;\frac{f(\GG^{\ipi};\zb;h)}{u-\ga_{\pci}}\,\;
\prod_{\satop{j=1}{j\ne i}}^{\la_p}\,\frac1{\ga_{\pcj}-\ga_{\pci}}\;
\prod_{k=1}^{\la_{p+1}\!}\,(\ga_{\pci}-\ga_{p+1,k}-h)\,\biggr]\,,
\eean
\vv-.3>
\bea
\rho\bigl(F_p(u)\bigr)\<\>:\>H^*_T(T^*\F_{\bla\<\>+\aal_p})\,\mapsto\,H^*_T(T^*\F_{\bla})\,,
\eea
\bean
\label{rho F}
\phantom{aaaaa}
\rho\bigl(F_p(u)\bigr)\<\>:\>[f]\;\mapsto\,\biggl[\;
\sum_{i=1}^{\la_{p+1}\!}\;\frac{f(\GG^{\ip};\zb;h)}{u-\ga_{\poi}}\,\;
\prod_{\satop{j=1}{j\ne i}}^{\la_{p+1}\!}\,\frac1{\ga_{\poi}-\ga_{\poj}}\;
\prod_{k=1}^{\la_p}\,(\ga_{p,k}-\ga_{\poi}-h)\,\biggr]\,,
\eean
where
\vvn-.5>
\begin{gather*}
\GG^{\ipi}=\,(\Ga_1\<\>\lsym;\Ga_{p-1}\<\>;\Ga_p-\{\ga_{\pci}\};
\Ga_{p+1}\cup\{\ga_{\pci}\};\Ga_{p+2}\<\>\lsym;\Ga_N)\,,
\\[8pt]
\GG^{\ip}=\,(\Ga_1\<\>\lsym;\Ga_{p-1}\<\>;\Ga_p\cup\{\ga_{\poi}\};
\Ga_{p+1}-\{\ga_{\poi}\};\Ga_{p+2}\<\>\lsym;\Ga_N)\,.
\\[-14pt]
\end{gather*}

\begin{thm} [Theorem 5.10 in \cite{GRTV}]
These formulas define a \>$\ty$-module structure on \;$H^*_T(\XX_n)$.
\end{thm}

This \>$\ty$-module structure was denoted in \cite{GRTV} by $\rho^-$ and $h$ in \cite{GRTV} is replaced with $-h$.
The topological interpretation of this \>$\ty$-action see in \cite[Theorem 5.16]{GRTV}.

\smallskip
In \cite{MO}, a Yangian module structure on $H^*_T(\XX_n)$ was introduced.

\begin{thm} [Corollary 6.4 in \cite{RTV}]
\label{yang str the same}
The $\ty$-module structure $\rho$ on $H^*_T(\XX_n)$ coincides with the Yangian module structure on
$H^*_T(\XX_n)$ introduced in \cite{MO}.

\end{thm}

\section{Dynamical Hamiltonians and quantum multiplication}
\label{DynHam}

\subsection{Dynamical Hamiltonians}
\label{sec dyn hams}
Assume that \,$\kkk$ \>are distinct numbers. Define the elements
\,$\Xk_1\lc\Xk_N\in\ty$ by the rule
\vvn.3>
\begin{align}
\label{X}
\Xk_i\, {}=\,-h\>T_{\ii}\+2+\>\frac h2\,e_{\ii}\>\bigl(\<\>e_{\ii}-1\bigr)-\>
h\>\sum_{\satop{j=1}{j\ne i}}^N\,\frac{\kk_j}{\kk_i-\kk_j}\,G_{\ij}\,,
\notag
\end{align}
where $G_{\ij}\>=\,e_{\ij}\>e_{\ji}\<-e_{\ii}\>=\,e_{\ji}\>e_{\ij}\<-e_{\jj}\>.$
By taking the limit $\kk_{i+1}/\kk_i\to0$ \,for all \,$i=1\lc N-1$,
we define the elements \,$\Xin_1\lc\Xin_N\in\ty$,
\vvn-.1>
\bea
\label{X8}
\Xin_i {}=\,-h\>T_{\ii}\+2+\>\frac h2\,e_{\ii}\>\bigl(\<\>e_{\ii}-1\bigr)
+h\>(G_{i,1}\<\lsym+\<\>G_{\ii-1})\,,
\eea
see \cite{GRTV}. The elements
\vvn.1>
\,$\Xk_i\<,\> \Xin_i$, \,$i=1\lc N$, are called the dynamical Hamiltonians.
Observe that
\vvn-.3>
\bea
\label{XX}
\Xk_i\,=\,\Xin_i-\>h\>\sum_{j=1}^{i-1}\,\frac{\kk_i}{\kk_i-\kk_j}\,G_{\ij}
\>-\>h\!\sum_{j=i+1}^n\frac{\kk_j}{\kk_i-\kk_j}\,G_{\ij}\,.
\eea
Given \,$\bla=(\la_1\lc\la_N)$, \,set
\vvn.1>
\;$G_{\bla\<\>,\<\>\ij}\>=\,e_{\ji}\>e_{\ij}$ \,for \,$\la_i\ge\la_j$\>
\,and
\;$G_{\bla\<\>,\<\>\ij}=\>e_{\ij}\>e_{\ji}$ \,for \,$\la_i<\la_j$\>.
We define the elements
\,$X^q_{\bla\<\>,1}\,\lc X^q_{\bla\<\>,\<\>N}\<\in\ty$,
\bea
\label{Xkm}
X^q_{\bla\<\>,\<\>i}\>=\,\Xin_i-
\>h\>\sum_{j=1}^{i-1}\,\frac{\kk_i}{\kk_i-\kk_j}\,G_{\bla\<\>,\<\>\ij}\>-
\>h\!\sum_{j=i+1}^n\frac{\kk_j}{\kk_i-\kk_j}\,G_{\bla\<\>,\<\>\ij}\,.
\eea
Let $\kappa\in\C^\times$. The formal differential operators
\vvn-.4>
\beq
\label{dyneq}
\nabla_{\qq,\kp,i}\>=\,\kp\,\ddk_i\>-\>\Xk_i,\qquad i=1\lc N,
\eeq
pairwise commute and, hence, define a flat connection $\nabla_{\qq,\kp}$ for any \,$\ty$-module, see \cite{GRTV}.

\begin{lem}[Lemma 3.5 in \cite{GRTV}]
\label{flat+-}
The connection \;\;$\nabla_{\<\bla,\qq,\kp}\<$ defined by
\vvn.2>
\bean
\label{nablapm}
\nabla_{\<\bla\<\>,\qq,\kp,\<\>i}=\,\kp\,\ddk_i\>-\>X^q_{\bla\<\>,\<\>i}\,,
\vv-.1>
\eean
$i=1\lc N$, is flat for any \,$\kp$.
\end{lem}
\begin{proof}
The connection \;\;$\nabla_{\<\bla,\qq,\kp}\<$ is gauge
equivalent to connection $\nabla_\kp$,
\vvn.4>
\bean
\label{gauge}
\nabla_{\<\bla\<\>,\qq,\kp,\<\>i}=\,(\Ups_\bla)^{-1}\;\nabla_{\qq,\kp,i}\;\Ups_\bla\,,
\qquad \Ups_\bla\>=\prod_{1\le i<j\le N}\!(1-\kk_j/\kk_i)
^{\<\>h\>\eps_{\bla\<\>,\ij}/\<\kp}\,,
\vv.2>
\eean
where \;$\eps_{\bla\<\>,\ij}=\>e_{\jj}$ \,for \,$\la_i\ge\la_j$\>,
\,and \;$\eps_{\bla\<\>,\ij}=\>e_{\ii}$ \,for \,$\la_i<\la_j$\>.
\end{proof}

Connection \Ref{dyneq} was introduced in \cite{TV4}, see also Appendix B
in \cite{MTV1}, and is called the {\it trigonometric dynamical connection\/}.
Later the definition was extended from $\frak{sl}_N$ to other simple Lie
algebras in \cite{TL} under the name of the trigonometric Casimir connection.

The trigonometric dynamical connection is
defined over \;$\C^N$ with coordinates \,$\kkk$, it has singularities
at the union of the diagonals \,$\kk_i=\kk_j$.

\subsection{Dynamical Hamiltonians $X^q_{\bla,i}$ on $H^*_T(T^*\F_\bla)$} Recall the $\ty$-module structure
$\rho$ defined on $H^*_T(\XX_n) = \oplus_{|\bla|=n}H^*_T(T^*\F_\bla)$ in Section \ref{sec cohom and Yang}.
For any $\bs\mu=(\mu_1,\dots,\mu_N)\in \Z^N_{\geq 0}$, $|\bs\mu|=n$, the
action of the dynamical Hamiltonians $X^q_{\bs\mu,i}$ preserve each of $H^*_T(T^*\F_\bla)$.

\begin{lem}
[Lemma 7.6 in \cite{RTV}]
For any $\bla$ and $i=1,\dots,n$, the restriction of $\rho(X^q_{\bla,i})$ to
$H^*_T(T^*\F_\bla)$ has the form:
\bean
\label{dyn bla}
&&
\phantom{aaa}
\\
\rho(X^{q}_{\bla,i})
&=&
(\gamma_{i,1}+\dots+\ga_{i,\la_i})
-\>h\>\sum_{j=1}^{i-1}\,\frac{\kk_i}{\kk_i-\kk_j}\,\rho(G_{\bla\<\>,\<\>\ij})
\>-\>h\!\sum_{j=i+1}^n\frac{\kk_j}{\kk_i-\kk_j}\,\rho(G_{\bla\<\>,\<\>\ij}) =
\notag
\\
&=& (\gamma_{i,1}+\dots+\ga_{i,\la_i})
-\>h\>\sum_{j=1}^{i-1}\,\frac{\kk_i}{\kk_i-\kk_j}\,\rho(e_{j,i}e_{i,j})
\>-\>h\!\sum_{j=i+1}^n\frac{\kk_j}{\kk_i-\kk_j}\,\rho(e_{i,j}e_{j,i})\, + C,
\notag
\eean
where $(\gamma_{i,1}+\dots+\ga_{i,\la_i})$ denotes the operator of multiplication by the cohomology class
$\gamma_{i,1}+\dots+\ga_{i,\la_i}$, the operator
$C$ is a scalar operator on $H^*_T(T^*\F_\bla)$, and for any $i\ne j$ the element
$\rho(G_{\bla\<\>,\<\>\ij})$ annihilates the identity element $1_\bla\in H^*_T(T^*\F_\bla)$.

\end{lem}

\subsection{Quantum multiplication by divisors on $H^*_T(T^*\F_\bla)$} In \cite{MO}, the quantum multiplication
by divisors on $H^*_T(T^*\F_\bla)$ is described. The fundamental equivariant cohomology classes of divisors
on $T^*\F_\bla$ are linear combinations of $D_{i}=\gamma_{i,1}+\dots+\ga_{i,\la_i}$, $i=1,\dots,N$.
The quantum multiplication $D_{i}*_{\tilde\qq}$ depends on parameters $\tilde\qq=(\tilde q_1,\dots,\tilde q_N)\in(\C^\times)^N$.

\begin{thm}[Theorem 10.2.1 in \cite{MO}]
\label{MO main} For $i=1,\dots,N$, the quantum multiplication by $D_i$ is given by the formula:
\bean
\label{q mult}
&&
\phantom{aaa}
\\
D_i*_{\tilde\qq}
&=& (\gamma_{i,1}+\dots+\ga_{i,\la_i})
+\>h\>\sum_{j=1}^{i-1}\,\frac{\tilde q_j/\tilde q_i}{1-\tilde q_j/\tilde q_i}\,\rho(e_{j,i}e_{i,j})
\>-\>h\!\sum_{j=i+1}^n\frac{\tilde q_i/\tilde q_j}{1-\tilde q_i/\tilde q_J}\,\rho(e_{i,j}e_{j,i})\, + C,
\notag
\eean
where $C$ is a scalar operator on $H^*_T(T^*\F_\bla)$ fixed by the requirement that
the purely quantum
part of $D_i*_{\tilde\qq}$ annihilates the identity $1_\bla$.

\end{thm}

\begin{cor} [Corollary 7.8 in \cite{RTV}]
\label{cor qm = dh}
For $i=1,\dots,N$, the operator $D_i*_{\tilde\qq}$
of quantum multiplication by $D_i$ on $H^*_T(T^*\F_\bla)$
equals the action $\rho(X^{q}_{\bla,i})$ on $H^*_T(T^*\F_\bla)$
of the dynamical Hamiltonian $X^{q}_{\bla,i}$ if we put
$(q_1,\dots,q_N)$ $ = (\tilde q_1^{\,-1},\dots, \tilde q_N^{\,-1})$.

\end{cor}

The quantum connection $\nabla_{\on{quant},\bla,\tilde\qq,\kp}$ on $H^*_T(T^*\F_\bla)$ is defined by the formula
\bean
\label{q conn}
\nabla_{\on{quant},\bla,\tilde \qq,\kp, i}\,=\,\kp\,\tilde q_i\frac{\der}{\der \tilde q_i}\>-\>D_i*_{\tilde\qq}\,,\qquad i=1\lc N,
\vv-.1>
\eean
where $\kp\in\C^\times$ is a parameter of the connection, see \cite{BMO}.
By Corollary \ref{cor qm = dh}, we have
\bean
\label{quanT}
\nabla_{\on{quant},\bla,\tilde \qq,\kp, i} = \rho(\nabla_{\bla,\tilde q_1^{\,-1},\dots, \tilde q_N^{\,-1},-\kp}),
\qquad i=1,\dots,N.
\eean

\subsection{Dynamical Hamiltonians on $(\C^N)^{\otimes n}\otimes \C[\zz;h]$}

Recall that $e_{\ij}$, $i,j=1\lc N$, denote standard generators of
$\gln$.
A vector $v$ of a $\gln$-module $M$ has weight
$\bla=(\la_1\lc\la_N)\in\C^N$ if $e_{\ii}\>v=\la_i\>v$ for $i=1\lc N$.
We denote by $M_\bla\subset M$ the weight subspace of weight $\bla$.

\vsk.2>
We consider $\C^N$ as the standard vector representation of $\gln$ with basis
$v_1\lc v_N$ such that $e_{\ij}v_k=\dl_{\jk}v_i$ for all $i,j,k$.
Denote $V=(\C^N)^{\ox n}$. For $I=(I_1,\dots,I_N)\in\mc I_n$, we define
$v_I\in V$ by the formula
$v_I = v_{i_1}\otimes\dots\otimes v_{i_n},$ where $i_j =i$ if $ i_j\in I_i$.
Let
\bea
V\,=\!\bigoplus_{|\bla|=n}\!V_\bla
\eea
be the weight decomposition. The vectors $(v_I)_{I\in\Il}$ form a basis of
$V_\bla$.

As always, we denote by $e^{(a)}_{i,j}$ the action of $e_{i,j}$ on the $a$-th tensor factor of $V$ and denote
$e_{i,j} = \sum_{a=1}^n e^{(a)}_{i,j}$.

Recall the $\ty$-module structure $\pho$ on $V\otimes \C[\zz;h]$ and the dynamical Hamiltonians $X^q_i$ introduced in
Section \ref{sec dyn hams}. The dynamical Hamiltonians preserve each of the weight subspaces
$V_\bla\otimes \C[\zz;h]$.

\begin{lem} [Lemma 4.17 in \cite{GRTV}]
\label{lem action on V}
For $i=1,\dots,n$, we have
\bean
\label{phoX8}
\phantom{aaaa}
\pho(X^q_i)
\, =\,\sum_{a=1}^n z_a\>e^{(a)}_{\ii}+\>
\frac h2\>(e_{\ii}^2-e_{\ii})-\>
h\>\sum_{j=1}^N\,\sum_{1\leq a<b\leq n}\!e^{(a)}_{\ij}\>e^{(b)}_{\ji}
-\>h\>\sum_{\satop{j=1}{j\ne i}}^N\,\frac{\kk_j}{\kk_i-\kk_j}\,G_{\ij}\,.
\eean
\end{lem}

To obtain the lemma we replace $h$ with $-h$ in Lemma 4.17 of \cite{GRTV}.

\subsection{\qKZ/ difference connection}
\label{sec qkz}
Recall the $\ty$-action $\pho$ on $(\C^N)^{\otimes n}\otimes\C[\zz;h]$ introduced in
Section \ref{sec yang act C^N}. \,Let
\vvn.1>
\be
R^{(\ij)}(u)\,=\,\frac{u-h\<\>P^{(\ij)}}{u-h}\;,\qquad
i,j=1\lc n\,,\quad i\ne j\,.\kern-3em
\vv.3>
\ee
For $\kp\in\C^\times$, define operators
\,$\Ko_1\lc\Ko_n\in\End( (\C^N)^{\otimes n})\ox\C[\zz;h]$\>,
\vvn.2>
\begin{align*}
\Ko_i(\zz;\qq;h;\kp)\>&{}=\,
R^{(i+1,i)}(z_{i+1}\<-z_i)\,\dots\,R^{(n,i)}(z_n\<-z_i)\,\times{}
\\[2pt]
& {}\>\times\<\;q_1^{-e_{1,1}^{(i)}}\!\dots\,q_N^{-e_{N,N}^{(i)}}\,
R^{(1,i)}(z_1\<-z_i\<-\kp\<\>)\,\dots\,R^{(i-1,i)}(z_{i-1}\<-z_i\<-\kp\<\>)\,.
\end{align*}
Consider the difference operators \,$\Kh_{\kp,1}\lc\Kh_{\kp,n}$
\vvn.2>
acting on \,$ (\C^N)^{\otimes n}\<$-valued functions of \>$\zz,\qq, h$\>,
\beq
\label{K_i}
\Kh_{\zz,\qq,h,\kp,i}\>F(\zz;\qq;h)\,=\,
K_i(\zz;\qq;h;\kp)\,F(z_1,\dots,z_{i-1},z_i+\kp,z_{i+1},\dots,z_n;\qq;h)\,.
\eeq

\begin{thm}[\cite{FR}]
The operators \;$\Kh_{\zz,\qq,h,\kp,1}\lc\Kh_{\zz,\qq,h,\kp,n}$ pairwise commute.
\end{thm}

\begin{thm}[\cite{TV4}]
\label{thm qkz}
The operators \;$\Kh_{\zz,\qq,h,\kp,1}\lc\Kh_{\zz,\qq,h,\kp,n}$,
\,$\pho(\nabla_{\bla,\qq,\kp,1})\lc\pho(\nabla_{\bla,\qq,\kp,N})$ \,pairwise commute.
\end{thm}

The commuting difference operators $\Kh_{\zz,\qq,h,\kp,1}\lc\Kh_{\zz,\qq,h,\kp,n}$ define the
rational \qKZ/ difference connection. We say that a $(\C^N)^{\otimes n}$-valued
function $F(\zz;\qq;h)$ is a flat section of the difference connection if
\bea
\Kh_{\zz,\qq,h,\kp,i} F(\zz;\qq;h) = F(\zz;\qq;h),
\qquad
i=1,\dots,n.
\eea
Theorem \ref{thm qkz} says that the \qKZ/ difference connection commutes with
the trigonometric dynamical connection $\pho(\nabla_{\bla,\qq,\kp})$ .

\section{Yangian $\ty$ weight functions}
\label{sec Weight functions}

\subsection{Weight functions $W_{I}$}
For $I\in \Il$, we define the weight functions $W_{I}(\TT;\zz;h)$, c.f. \cite{TV1, TV5}.

Recall \,$\bla=(\la_1\lc\la_N)$. Denote \,$\la^{(i)}\>=\la_1\lsym+\la_i$ and
\,$\la^{\{1\}}\<=\sum_{i=1}^{N-1}\la^{(i)}=$
\linebreak $ \sum_{i=1}^{N-1}(N\<\<-i)\>\la_i$\>.
Recall $I=(I_1,\dots,I_N)$. Set
\;$\bigcup_{\>k=1}^{\,j}I_k=\>\{\>i^{(j)}_1\!\lsym<i^{(j)}_{\la^{(j)}}\}$\>. Consider the variables
\,$t^{(j)}_a$, \,$j=1\lc N$, \,$a=1\lc\la^{(j)}$,
where \,$t^{(N)}_a=z_a$, \,$a=1\lc n$\>. Denote $t^{(j)}=(t^{(j)}_k)_{k\leq\la^{(j)}}$ and \,$\TT=(t^{(1)}, \dots, t^{(N-1)})$.

\vsk.2>
The weight functions are
\vvn.4>
\beq
\label{hWI-}
W_I(\TT;\zb;h)\,=\,(-h)^{\>\la^{\{1\}}}\,
\Sym_{\>t^{(1)}_1\!\lc\,t^{(1)}_{\la^{(1)}}}\,\ldots\;
\Sym_{\>t^{(N-1)}_1\!\lc\,t^{(N-1)}_{\la^{(N-1)}}}U_I(\TT;\zb;h)\,,
\vv.3>
\eeq
\be
U_I(\TT;\zb;h)\,=\,\prod_{j=1}^{N-1}\,\prod_{a=1}^{\la^{(j)}}\,\biggl(
\prod_{\satop{c=1}{i^{(j+1)}_c\<<\>i^{(j)}_a}}^{\la^{(j+1)}}
\!\!(t^{(j)}_a\<\<-t^{(j+1)}_c-h)
\prod_{\satop{d=1}{i^{(j+1)}_d>\>i^{(j)}_a}}^{\la^{(j+1)}}
\!\!(t^{(j)}_a\<\<-t^{(j+1)}_d )\,\prod_{b=a+1}^{\la^{(j)}}
\frac{t^{(j)}_a\<\<-t^{(j)}_b\<\<-h}{t^{(j)}_a\<\<-t^{(j)}_b}\,\biggr)\,.
\ee
In these formulas for a function $f(t_1,\dots,t_k)$ of some variables we denote
\be
\Sym_{t_1,\dots,t_k}f(t_1,\dots,t_k) = \sum_{\sigma\in S_k}f(t_{\sigma_1},\dots,t_{\sigma_k}).
\ee

\begin{example}
Let $N=2$, $n=2$, $\bla=(1,1)$, $I=(\{1\},\{2\})$,
$J=(\{2\}, \{1\})$. Then
\bea
W_I(\TT;\zz;h)= -h\, (t^{(1)}_1\<\<-z_2),
\qquad
W_J(\TT;\zz;h)= -h\, (t^{(1)}_1\<\<-z_1-h).
\eea
\end{example}

\subsection{Weight functions $W_{\si,I}$}

For $\si\in S_n$ and $I\in\Il$, we define
\bea
W_{\si,I}(\TT;\zz;h) = W_{\si^{-1}(I)}(\TT;z_{\si(1)},\dots,z_{\si(n)};h),
\eea
where $\si^{-1}(I)=(\si^{-1}(I_1),\dots,\si^{-1}(I_N))$.

For a subset $A\subset\{1,\dots,n\}$, denote $\zz_A=(z_a)_{a\in A}$. For $I\in\Il$, denote
$\zz_I=(\zz_{I_1},\dots,\zz_{I_N})$. For
$f(t^{(1)},\dots,t^{(N)})\in\C[t^{(1)},\dots,t^{(N)}]^{S_{\la^{(1)}}\times\dots\times S_{\la^{(N)}}}$,
we define $f(\zz_I)$ by replacing $t^{(j)}$ with $\cup_{k=1}^j\zz_{I_k}$.
Denote
\bea
c_\bla(\zb_I)\,=\,
\prod_{a=1}^{N-1}\>\prod_{\ij\in \cup_{b=1}^a I_b\!}\>(z_i\<-z_j\<-h)\,,
\qquad
\vv.1>
\eea
\bea
\label{RQ}
R(\zb_I)\,=\!\prod_{1\le a<b\le N}\,\prod_{i\in I_a}\,\prod_{j\in I_b}\,
(z_i-z_j)\,,\qquad
Q(\zb_I)\,=\!\prod_{1\le a<b\le N}\,\prod_{i\in I_a}\,\prod_{j\in I_b}\,
(z_i-z_j-h)\,.
\vv.3>
\eea

\subsection{Stable envelope map}
Following \cite{RTV}, we define the weight function map
\bea
[W_{\on{id}}] : V\otimes \C[\zz;h] \to H_T^*(\XX_n) , \quad
v_I \mapsto [W_{\on{id},I}(\bs\Theta;\zz;h)],
\eea
where $W_{\on{id},I}(\bs\Theta;\zz;h)$ is the polynomial $W_{\on{id},I}(\TT;\zz;h)$ in which variables
$t^{(j)}_i$ are replaced with $\theta_{j,i}$ and $[W_{\on{id},I}(\bs\Theta;\zz;h)]$ is the cohomology class represented by
$W_{\on{id},I}(\bs\Theta;\zz;h)$.

\smallskip
Denote
\bea
c_\bla(\bs \Theta) = \prod_{a=1}^{N-1}\prod_{i=1}^{\la^{(a)}}\prod_{j=1}^{\la^{(a)}}(\theta_{a,i}-\theta_{a,j} - h) \in H^*_T(T^*\F_\bla).
\eea
Observe that $c_\bla(\bs \Theta)$ is the equivariant Euler class of the bundle
$\oplus_{a=1}^{N-1} \Hom(\bs F_a, \bs F_a)$ if we make $\C^\times$ act on it with weight $-h$.
Note that $c_\bla(\bs \Theta)$ is not a zero-divisor in $H^*_T(\XX_n)$,
because none of its fixed point restrictions is zero.

\begin{thm} [Theorem 4.1 in \cite{RTV}] For any $\bla$ and any $I\in\Il$, the cohomology class
$[W_{\on{id},I}(\bs\Theta;\zz;h)]\in H^*_T(T^*\F_\bla)$ is divisible by $c_\bla(\bs \Theta)$, that is, there exists
a unique element
$\St_{\on{id},I} \in H^*_T(T^*\F_\bla)$ such that
\bean
\label{StW}
[W_{\on{id},I}(\bs\Theta;\zz;h)]\,=\, c_\bla(\bs \Theta)\cdot\St_{\on{id},I}\,.
\eean

\end{thm}

By using this theorem we define the stable envelope map
\bea
\St_{\on{id}} : V\otimes \C[\zz;h] \to H_T^*(\XX_n) , \quad
v_I \mapsto \St_{\on{id},I}.
\eea
The stable envelope maps are main objects in \cite{MO}.
The stable envelope maps are defined in \cite{MO} in terms of the torus $T$ action on $\XX_n$.
Relation \Ref{StW} gives a formula for the stable envelope map in terms of the Chern roots $\bs\Theta,\zz, h$.

As we know, formula \Ref{pho} defines the $\ty$-module structure $\pho$ on
$V\otimes \C[\zz;h]$, and formulas \Ref{Arho}, \Ref{rho E}, \Ref{rho F} define
the $\ty$-module structure $\rho$ on $H^*_T(\XX_n)$.

\begin{thm}[Theorem 6.3 in \cite{RTV}]
\label{thm stab yang}
The stable envelope map $\St_{\on{id}} : V\otimes \C[\zz;h]\to H^*_T(\XX_n)$ is a homomorphism
of $\ty$-modules.
\end{thm}

\subsection{The inverse map}
For $I\in\Il$, introduce $\xi_{I}\in \Czh$ by the formula
\bean
\label{xi}
\xi_{I} = \sum_{J\in \Il} \frac{W_{\si_0,J}(\zz_I;\zz;h)}{Q(\zz_I)\,c_\bla(\zz_I)} \,v_J ,
\eean
where $\si_0\in S_n$ is the longest permutation.

Let $\C(\zz;h)$ be the algebra of rational functions in $\zz, h$. Consider the map
\bea
\nu=\oplus_{|\bla|=n}\nu_\bla \
: \ \oplus_{|\bla|=n} H_T^*(\XX_\bla)\otimes\C(\zz;h)\ \to \
\Czh,
\eea
where $\nu_\bla$ is defined by the formula
\bea
[f(\bs\Theta;\zz;h)] \mapsto \sum_{I\in\Il} \frac{f(\zz_I;\zz;h)}{R(\zz_I)}\xi_{I}\,,
\eea
see \cite[Formula (5.9)]{GRTV}.

\begin{lem} [Lemma 6.7 in \cite{RTV}]
\label{lem nu stab = 1}
The operator $\nu$ is inverse to $\St_{\on{id}}$.
\end{lem}

\subsection{The difference connection on $H^*_T(T^*\F_\bla)$}

By Theorem \ref{thm qkz}, the difference operators
\bea
\St_{\on{id}}\circ\Hat\Ko_{\zz,\tilde q_1^{\,-1},\dots, \tilde q_N^{\,-1},-\kp,1} \circ \nu,
\quad
\dots,
\quad
\St_{\on{id}}\circ\Hat\Ko_{\zz,\tilde q_1^{\,-1},\dots, \tilde q_N^{\,-1},-\kp,n} \circ\nu
\eea
and the differential operators $\nabla_{\on{quant},\bla,\tilde\qq,\kp,1},\dots,\nabla_{\on{quant},\bla,\tilde\qq,\kp,N}$
pairwise commute. The difference operators form
the rational \qKZ/ difference connection on $H^*_T(T^*\F_\bla)$. This difference connection is discussed in \cite{MO}
under the name of the shift operators.

\section{Four more connections}

In the remainder of this paper we fix $\bla\in\Z_{\geq 0}^N$, $|\bla|=n$.

\subsection{Trigonometric dynamical connection on $\pi_{V_\bla}$}
Consider $\C^n$ with coordinates $\xx=(x_1,\dots,x_n)$
and $\C^N$ with coordinates $\qq=(q_1,\dots,q_N)$.
Consider
the trivial bundle $\pi_{V_\bla} : V_\bla \times \C^{n+N} \to \C^{n+N}$.

Following \cite{TV4} and \cite[Appendix B]{MTV1},\ introduce dynamical Hamiltonians
\bea
X_i^{V_\bla}(\xx,\qq) \, =\,\sum_{a=1}^n x_a\>e^{(a)}_{\ii}-\>
\frac {e_{\ii}^2}2+\>\sum_{j=1}^N\,\sum_{1\leq a<b\leq n}\!e^{(a)}_{\ij}\>e^{(b)}_{\ji}
+\>\sum_{\satop{j=1}{j\ne i}}^N\,\frac{q_j}{q_i-q_j}\,G_{\ij}\,,
\eea
acting the $V_\bla$-valued functions of $\xx,\qq$. By formula \Ref{phoX8}, we have
\bean
\label{HH}
\pho(X^q_i)(\zz; \qq; h)
&=& -h X_i^{V_\bla}(-z_1/h,\dots,-z_n/h;\qq)- \frac h2 \,e_{i,i}\,=
\\
\notag
&=& -h X_i^{V_\bla}(-z_1/h,\dots,-z_n/h;\qq)- \frac h2 \,\la_i\,.
\eean
For $\kp\in\C^\times$ the differential operators
\bean
\label{t nabla}
\nabla_{\xx,\qq,\kp,i}^{V_\bla}\>=\,\kp\,q_i\frac{\der}{\der q_i}\>-\>X_i^{V_\bla}(\xx,\qq), \qquad i=1\lc N,
\eean
pairwise commute and define a flat connection $\nabla_{\xx,\qq,\kp}^{V_\bla}$ on
$\pi_{V_\bla}$, see \cite{TV4}.

\subsection{Rational \qKZ/ difference connection on $\pi_{V_\bla}$}
\label{sec qkZ} \,Let
\vvn.1>
\be
\tilde R^{(\ij)}(u)\,=\,\frac{u+\<\>P^{(\ij)}}{u+1}\;,\qquad
i,j=1\lc n\,,\quad i\ne j\,.\kern-3em
\vv.3>
\ee
For $\kp\in\C^\times$, define operators \,$K_1^{V_\bla} \lc \tilde K_n^{V_\bla} $ on $V_\bla$ by the formula
\vvn.4>
\begin{align*}
K_i^{V_\bla} (\xx;\qq;\kp)\>&{}=\,
\tilde R^{(i+1,i)}(x_{i+1}\<-x_i)\,\dots\,\tilde R^{(n,i)}(x_n\<-x_i)\,\times{}
\\[2pt]
& {}\>\times\<\;q_1^{-e_{1,1}^{(i)}}\!\dots\,q_N^{-e_{N,N}^{(i)}}\,
\tilde R^{(1,i)}(x_1\<-x_i\<-\kp\<\>)\,\dots\,
\tilde R^{(i-1,i)}(x_{i-1}\<-x_i\<-\kp\<\>)\,.
\end{align*}
Consider the difference operators
\,$K_{\xx,\qq,\kp,1}^{V_\bla},\dots,K_{\xx,\qq,\kp,n}^{V_\bla}$
\vvn.2>
acting on \,$ (\C^N)^{\otimes n}\<$-valued functions of \>$\xx,\qq$\>,
\beq
\label{bar K_i}
\Kh_{\xx,\qq,\kp,i}^{V_\bla}\>F(\xx,\qq)\,=\,
K_i^{V_\bla}(\xx;\qq;\kp)\,F(x_1,\dots,x_{i-1},x_i+\kp,x_{i+1},\dots,x_n;\qq)\,.
\eeq

\begin{thm}[\cite{FR}]
The operators \;$\Kh_{\xx,\qq,\kp,1}^{V_\bla}\lc\Kh_{\xx,\qq,\kp,n}^{V_\bla}$ pairwise commute.
\end{thm}

\begin{thm}[\cite{TV4}]
\label{thm qkz2}
The operators \;$\Kh_{\xx,\qq,\kp,1}^{V_\bla}\lc\Kh_{\xx,\qq,\kp,n}^{V_\bla}$\>,
\,$\nabla_{\xx,\qq,\kp,1}^{V_\bla}, \dots, \nabla_{\xx,\qq,\kp,N}^{V_\bla}$ \,pairwise commute.
\end{thm}

The commuting difference operators
$\Kh_{\xx,\qq,\kp,1}^{V_\bla}\lc\Kh_{\xx,\qq,\kp,n}^{V_\bla}$ define the rational
\qKZ/ difference connection on $\pi_{V_\bla}$. Theorem \ref{thm qkz2} says
that the rational \qKZ/ difference connection commutes with the trigonometric
dynamical connection $\nabla_{\xx,\qq,\kp}^{V_\bla}$.

\subsection{Relations between flat sections}

\begin{lem}
\label{ttt}
Let $\kp,h\in\C^\times$.
\begin{enumerate}
\item[(i)] Assume
that $\tilde F(\xx;\qq) \in V_\bla$ is a flat section of the connection $\nabla_{\xx,\qq,-\kp/h}^{V_\bla}$.
Then
\bea
F(\zz;\qq;h;\kp)=
\prod_{i=1}^N q_i^{-\frac{h\la_i}{2\kp}}
\prod_{1\le i<j\le N}\!(1-\kk_j/\kk_i)
^{\<-\frac{h\>\min(\la_i,\la_j)}{\<\kp}} \tilde F(-z_1/h,\dots,-z_n/h,\qq)
\vv.2>
\eea
is a flat section of the connection $\pho(\nabla_{\bla,\qq,\kp})$.
\item[(ii)]
Assume that $\tilde F(\xx;\qq) \in V_\bla$ is a flat section
of the difference connection defined by the operators
$\Kh_{\xx,\qq,-\kp/h,1}^{V_\bla}\lc\Kh_{\xx,\qq,-\kp/h,n}^{V_\bla}$ in \Ref{bar K_i}.
Then $F(\zz;\qq;h)$ is a flat section of the difference connection defined
by the operators $\Kh_{\zz,\qq,\kp,1}\lc\Kh_{\zz,\qq,\kp,n}$ in \Ref{K_i}.
\end{enumerate}
\end{lem}

\begin{proof} Part (i) follows from formulas \Ref{gauge} and \Ref{HH}.
Part (ii) follows from formulas for the operators $\Kh_{\xx,\qq,-\kp/h,i}^{V_\bla}$
and $\Kh_{\zz,\qq,\kp,i}$.
\end{proof}

\begin{lem}
\label{V quant}
Assume
$F(\zz;\tilde \qq,h)\,=\,\sum_{I\in\Il}F_I(\zz;\tilde \qq,h)\,v_I\ \in V_\bla$
is a flat section of the connection $\pho(\nabla_{\bla,\tilde q^{-1}_1,\dots,\tilde q^{-1}_N,-\kp})$
(resp. of the difference connection defined by the operators $\Kh_{\zz,\tilde q^{-1}_1,\dots,\tilde q^{-1}_N,-\kp,1}\lc\Kh_{\zz,\tilde q^{-1}_1,\dots,\tilde q^{-1}_N,-\kp,n}$). Then
\bean
\phantom{aaa}
\St_{\on{id}} (F(\zz;\tilde\qq,h))=\sum_{I\in\Il} F_I(\zz;\tilde\qq;h) \,\St_{\on{id},I} = \sum_{I\in\Il} F_I(\zz;\tilde\qq;h) \,\frac{[W_{\on{id},I}(\bs\Theta;\zz;h)]}{c_\bla(\bs\Theta)}
\eean
is a flat section of the quantum connection $\nabla_{\on{quant},\bla,\tilde\qq,\kp}$
(resp. of the difference connection defined by the operators $\St_{\on{id}}\circ \Kh_{\zz,\tilde q^{-1}_1,\dots,\tilde q^{-1}_N,-\kp,1}
 \circ \nu,\dots,$ $\St_{\on{id}}\circ\Kh_{\zz,\tilde q^{-1}_1,\dots,\tilde q^{-1}_N,-\kp,n} \circ\nu$).

\end{lem}
\begin{proof} The first statement follows from Corollary \ref{cor qm = dh}. The second statement of obvious.
\end{proof}

\subsection{Trigonometric KZ connection on $W_\bla$} Consider the Lie algebra $\glnn$.
For $m\in\Z_{\geq 0}$, we denote by $W_m^{(n)}$ the irreducible $\glnn$-module with highest weight $(m,0,\dots,0)$
and highest weight vector $w_m$. For given $\bla=(\la_1,\dots,\la_N)$, $|\bla|=\la_1+\dots+\la_N=n$, we consider the $\glnn$-module
$\otimes_{i=1}^NW_{\la_i}^{(n)}$. We denote by $W_\bla$ the weight subspace of $\otimes_{i=1}^NW_{\la_i}^{(n)}$ of
$\glnn$-weight $(1,\dots,1)$. We consider the trivial bundle $\pi_{W_\bla} : W_\bla \times \C^{n+N} \to \C^{n+N}$.

Let $\Omega_0=\frac 12\sum_{a=1}^n e_{a,a}\otimes e_{a,a}$, $\Omega_+ =\Omega_0+\sum_{1\leq a<b\leq n} e_{a,b}\otimes e_{b,a},$
$\Omega_- =\Omega_0+\sum_{1\leq a<b\leq n} e_{b,a}\otimes e_{a,b}.$
Introduce the trigonometric \KZ/ operators $X^{W_\bla}_1,\dots,X^{W_\bla}_N$ acting on $W_\bla$-valued functions of $\xx,\qq$,
\bea
X^{W_\bla}_i(\xx;\qq)=\sum_{a=1}^n(x_a- \frac {e_{a,a}}2) e_{a,a}^{(i)}
+ \sum_{\satop{j=1}{j\ne i}}^N
\frac{q_i\Omega_+^{(i,j)} + q_j\Omega_-^{(i,j)}}{q_i-q_j}\,.
\eea
For $\kappa\in\C^\times$, introduce the differential operators
$\nabla_{\xx,\qq,\kp,1}^{W_\bla},\dots, \nabla_{\xx,\qq,\kp,N}^{W_\bla}$,
\bean
\label{kz}
\nabla_{\xx,\qq,\kp,i}^{W_\bla} = \kappa q_i\frac\der{\der q_i} - X^{W_\bla}_i(\xx;\qq)\,.
\eean
The differential operators define on $\pi_{W_\bla}$ a flat connection
$\nabla_{\xx,\qq,\kp}^{W_\bla}$ called the trigonometric KZ connection.

\subsection{Rational dynamical difference connection on $\pi_{W_\bla}$}

For any $a,b=1,\dots,n$, $a\ne b$, \>introduce a series $B_{a,b}(t)$ depending
on a variable $t$,
\bea
B_{a,b}(t) = 1+\sum_{s=1}^\8 (e_{b,a})^s(e_{a,b})^s \prod_{j=1}^s\frac 1{j(t-e_{a,a}+e_{b,b}-j)}\,.
\eea
The series acts on any finite-dimensional $\frak{gl}_n$-module $W$.

For $\kp\in\C^\times$, introduce the operators
$K^{W_\bla}_1,\dots,K^{W_\bla}_n$ acting on $W_\bla$-valued functions
\vvn.2>
of $\xx,\qq$,
\begin{align*}
K^{W_\bla}_i(\xx;\qq;\kp)\,&{}=\,
\big(B_{i,n}(x_i\<-x_{n})\,\dots\,B_{i,i+1}(x_i\<-x_{i+1})\big)^{-1}\times{}
\\[2pt]
&{}\>\times\,q_1^{-e_{\ii}^{(1)}}\!\dots\,q_N^{-e_{\ii}^{(N)}}\,
B_{1,i}(x_1-x_i-\kp)\,\dots\, B_{i-1,i}(x_{i-1}-x_i-\kp)\,.
\\[-14pt]
\end{align*}
Introduce the difference operators
$\Kh^{W_\bla}_{\xx,\qq,\kp,1},\dots,\Kh^{W_\bla}_{\xx,\qq,\kp,n}$\,,
\vvn.3>
\beq
\label{rdcw}
\Kh^{W_\bla}_{\xx,\qq,\kp,i} F(\xx,\qq)\,=\,
K^{W_\bla}_i(\xx,\qq,\kp)\,F(x_1,\dots,x_{i-1},x_i+\kp,x_{i+1},\dots,x_n;\qq).
\vv.2>
\eeq

\begin{thm}[\cite{TV3}]
\label{RDD}
The operators $\nabla_{\xx,\qq,\kp,1}^{W_\bla},\dots, \nabla_{\xx,\qq,\kp,N}^{W_\bla}$\,,
\,$\Kh^{W_\bla}_{\xx,\qq,\kp,1},\dots,\Kh^{W_\bla}_{\xx,\qq,\kp,n}$
pairwise commute.
\end{thm}

The commuting difference operators
\,$\Kh_{\xx,\qq,\kp,1}^{W_\bla}\lc\Kh_{\xx,\qq,\kp,n}^{W_\bla}$ define
the rational dynamical difference connection on $\pi_{W_\bla}$.
Theorem \ref{RDD} says that the rational dynamical difference connection
commutes with the trigonometric \KZ/ connection $\nabla_{\xx,\qq,\kp}^{W_\bla}$.

\subsection{Equivalence of connections on $\pi_{V_\bla}$ and $\pi_{W_\bla}$ }
\medskip
For $I=(I_1,\dots,I_N)\in\Il$, define a vector $w_I\in W_\bla$ by the formula
\bean
w_I=w_{I_1}\otimes \dots\otimes w_{I_N}, \qquad
w_{I_a}= \big(\prod_{i\in I_a}'e_{i,1}\big)\, w_{\la_a},
\eean
where $\prod'$ means that we exclude from the product the factor $e_{1,1}$ if $1\in I_a$.
The map
\bean
\mu : W_\bla \to V_\bla, \qquad w_I\mapsto v_I,
\eean
is a vector isomorphism, see for example \cite{TV4}.

\begin{thm} [Theorem 5.8 in \cite{TV4}]
\label{kz dyn thm 2}
The vector isomorphism $\mu$ identifies the action of operators
$\nabla_{\xx,\qq,\kp,1}^{V_\bla},\dots,\nabla_{\xx,\qq,\kp,n}^{V_\bla}$,
$\Kh^{V_\bla}_{\xx,\qq,\kp,1},\dots,\Kh^{V_\bla}_{\xx,\qq,\kp,n}$ with the action of
the operators $\nabla_{\xx,\qq,\kp,1}^{W_\bla},$ \dots, $\nabla_{\xx,\qq,\kp,n}^{W_\bla}$,
$Q_1\Kh^{W_\bla}_{\xx,\qq,\kp,1},\dots, Q_n\Kh^{W_\bla}_{\xx,\qq,\kp,n}$,
respectively, where
\be
Q_i\,=\,\prod_{1\leq j<i}\frac{x_j-x_i - 1-\kp}{x_j-x_i+1-\kp}
\prod_{i<j\leq n}\frac{x_i-x_j + 1}{x_i-x_j-1} .
\ee
\end{thm}

\begin{cor}
\label{cor kz dyn }
If a $W_\bla$-valued function $F^W(\xx;\qq)=\sum_{I\in\Il}F^W_I(\xx;\qq)w_I$
satisfies the equations
\be
\nabla_{\xx,\qq,\kp,1}^{W_\bla}F^W =0,\quad \dots \quad \nabla_{\xx,\qq,\kp,N}^{W_\bla}F^W=0,
\quad
\Kh^{W_\bla}_{\xx,\qq,\kp,1}F^W=F^W,\quad \dots \quad\Kh^{W_\bla}_{\xx,\qq,\kp,n}F^W=F^W ,
\ee
then the $V_\bla$-valued function
\be
F^V(\xx;\qq)= \prod_{1\leq i<j\leq n} \frac{\Gamma((x_i-x_j-1)/\kp)}{\Gamma((x_i-x_j+1)/\kp)} \sum_{I\in\Il} F^W_I(\xx;\qq)\,v_I
\ee
satisfies the equations
\be
\nabla_{\xx,\qq,\kp,1}^{V_\bla}F^V =0,\quad \dots \quad \nabla_{\xx,\qq,\kp,N}^{V_\bla}F^V =0,
\quad
\Kh^{V_\bla}_{\xx,\qq,\kp,1}F^V=F^V,\quad \dots \quad\Kh^{V_\bla}_{\xx,\qq,\kp,n}F^V=F^V .
\ee
\end{cor}

\section{Integral representations for flat sections}

\subsection{Master function}

Consider $\C^{\frac{n(n-1)}2}$ with coordinates $\ss =(s^{(a)}_j)$, $a=1,\dots,n-1$,
$j= 1,\dots,n-a$.
Consider $\C^n$ with coordinates $\xx=(x_1,\dots,x_n)$
and $\C^N$ with coordinates $\qq=(q_1,\dots,q_N)$.
The master function is
\vvn-.4>
\begin{align*}
\Phi_\bla(\ss;\xx;\qq)&{}=\,
\prod_{1\leq i<j\leq N}(q_i-q_j)^{\la_i\la_j}\,
\prod_{i=1}^N\,q_i^{\la_i(x_1-1+\la_i/2)}\,
\prod_{j=1}^{n-1}\,\prod_{i=1}^N\,(s^{(1)}_j-q_i)^{-\la_i}
\times
\\[4pt]
&{}\times\,\prod_{a=1}^{n-1}\,\prod_{i=1}^{n-a}\,(s^{(a)}_i)^{x_{a+1}-x_a+1}\,
\prod_{a=1}^{n-2}\big(\prod_{1\leq i<j\leq n-a} (s^{(a)}_i-s^{(a)}_j)^2\,
\prod_{i=1}^{n-a}\,\prod_{j=1}^{n-a-1}\!(s^{(a)}_i-s^{(a+1)}_j)^{-1}\big)\,,
\end{align*}
see \cite[Formula (16)]{MV}.

\subsection{$\frak{gl}_n$ weight functions}

For $I\in \Il$ we define the weight function $\omega_I(\ss;\qq)$, see \cite{MV, RSV, SV}.
Introduce new variables $\hat\ss= (s_{a,i,j})$, where
$a\in\{1,\dots,N\}$, $i\in I_a$, $j\in\{1,\dots,i-1\}$. For $k=1,\dots,n-1$, denote by
${\hat\ss}_k$ the set of all variables $s_{a,i,j}$ with $j=k$. Then
$|{\hat\ss}_k|=n-k$.

Let $B$ be the set of sequences \,$\bs\beta=(\beta_1,\dots,\beta_{n-1})$
of bijections \,$\beta_k:\hat\ss_k\to\{s^{(k)}_1,\dots,s^{(k)}_{n-k}\}$, \>$k=1,\dots,n-1$.

If $f(\hat\ss)$ is a function of $\hat \ss$ and $\bs\beta\in B$, then we obtain the function $(\bs\beta f)(\ss)$
of $\ss$ by replacing variables in $f(\hat\ss)$ according to the bijection $\bs\beta$.

Introduce the function
\bea
f_I(\hat\ss;\qq) =\prod_{a=1}^N \prod_{i\in I_a} \frac {1} {(s_{a,i,1}-q_a)(s_{a,i,2}-s_{a,i,1})\dots(s_{a,i,i-1}-s_{a,i,i-2})}\,,
\eea
where we set $\frac {1}{(s_{a,i,1}-q_a)(s_{a,i,2}-s_{a,i,1})\dots(s_{a,i,i-1}-s_{a,i,i-2})}=1$
if $i=1$. We define
\bea
\omega_I(\ss;\qq) =
\sum_{\bs\beta\in B}\,(\bs\beta f_I)(\ss;\qq) .
\eea

\begin{example}
Let $N=2$, $n=3$, $\bla=(1,2)$, $I=(I_1,I_2)$, $I_1=\{2\}, I_2=\{1,3\}$. Then
\bea
\omega_I(\ss;\qq) = \frac {1}{(s^{(1)}_1-q_1)(s^{(1)}_2-q_2)(s^{(2)}_1-s^{(1)}_2)}
+
\frac {1}{(s^{(1)}_2-q_1)(s^{(1)}_1-q_2)(s^{(2)}_1-s^{(1)}_1)}.
\eea
\end{example}

\subsection{Hypergeometric integrals}
\label{HINT}
In the space $\C^{\frac{n(n-1)}2}\times\C^n\times\C^N$
with coordinates $\ss,\xx,\qq$ we consider the arrangement ${\A}$ of hyperplanes
defined by equations
\bea
&&
s^{(a)}_i=0, \qquad a=1,\dots,n-1, \quad i=1,\dots,n-a;
\\
&&
s^{(1)}_j-q_i=0,\qquad j=1,\dots,n-1, \quad i=1,\dots,N;
\\
&&
s^{(a)}_i-s^{(a)}_j=0,\qquad a=1,\dots,n-1,\quad 1\leq i<j\leq n-a;
\\
&&
s^{(a)}_i-s^{(a+1)}_j=0, \qquad a=1,\dots,n-2, \quad i=1,\dots,n-a,\quad j=1,\dots,n-a-1.
\\
&&
q_i-q_j=0,\qquad 1\leq i<j\leq N;
\\
&&
q_i=0,\qquad i=1,\dots,N .
\eea
Denote by $ U$ the complement to the union of hyperplanes of the arrangement
$\A$.

In the space $\C^n\times\C^N$
with coordinates $\xx,\qq$ consider the arrangement of hyperplanes
defined by equations
\bea
&&
q_i-q_j=0,\qquad 1\leq i<j\leq N;
\\
&&
q_i=0,\qquad i=1,\dots,N .
\eea
Denote by $\Delta$ the complement to the union of hyperplanes of this arrangement.

Consider the projection $\pi : \C^{\frac{n(n-1)}2}\times\C^n\times\C^N
\to \C^n\times\C^N$. For every $(\xx;\qq)$ the arrangement $ \A$
induces an arrangement in the fiber of $\pi$ over $(\xx;\qq)$.
Denote by $ U(\xx;\qq)$ the complement to that arrangement in the fiber.

Consider the master function $\Phi_\bla(\ss;\xx; \qq)$ as a multivalued function on $U$.
Let $\kappa\in\C^\times$. The function $\Phi_\bla(\ss;\xx; \qq)^{1/\kappa}$
defines a rank one local system ${\mc L}_\kappa$ on $ U$ whose horizontal sections
over open subsets of $ U$ are univalued branches of $\Phi_\bla(\ss;\xx; \qq)^{1/\kappa}$
multiplied by complex numbers. The vector bundle
\bean
\label{GM hat}
\sqcup_{(\xx,\qq) \in \Delta}\,H_{\frac{n(n-1)}2}(U(\xx;\qq), {\mc L}_\kappa\vert_{U(\xx;\qq)})
\to \Delta
\eean
has the canonical flat Gauss-Manin connection.
Let
\bea
\psi(\xx,\qq,\kappa)
\in H_{\frac{n(n-1)}2}( U(\xx,\qq), {\mc L}_\kappa\vert_{U(\xx,\qq)})
\eea
be a  flat section of the Gauss-Manin connection.
In what follows we will consider the multidimensional hypergeometric integrals
\bea
{\mc J}_{\psi,I}(\xx;\qq;\kappa) = \int_{\psi(\xx,\qq,\kappa)}
\Phi_\bla(\ss;\xx; \qq)^{1/\kappa} \omega_I(\ss;\qq)
\wedge_{a=1}^{n-1}\wedge_{i=1}^{n-a} d s^{(a)}_i
\eea
for $I\in\Il$.

\begin{thm} [\cite{MV}]
\label{MV integral}
For every  flat section $\psi(\xx,\qq,\kappa)$ of the Gauss-Manin connection on the bundle in \Ref{GM hat},
the function
\bean
\label{main formula hat}
F_\psi(\xx;\qq;\kappa) = \sum_{I\in\Il} {\mc J}_{\psi,I}(\xx;\qq;\kappa)\, w_I
\eean
is a flat section of the trigonometric KZ connection \Ref{kz} on $\pi_{W_\bla}$ and of the rational dynamical
difference connection \Ref{rdcw} on $\pi_{W_\bla}$.
\end{thm}

\begin{cor}
\label{cor WW V}
The $V_\bla$-valued function
\bea
F^V_\psi(\xx;\qq;\kp)= \prod_{1\leq i<j\leq n} \frac{\Gamma((x_i-x_j-1)/\kp)}{\Gamma((x_i-x_j+1)/\kp)} \sum_{I\in\Il} {\mc J}_{\psi,I}(\xx;\qq;\kappa) \,v_I
\eea
is a flat section of the trigonometric dynamical connection \Ref{t nabla} on $\pi_{V_\bla}$ and of the rational \qKZ/
difference connection \Ref{bar K_i} on $\pi_{V_\bla}$.

\end{cor}

\begin{proof}
The corollary follows from Corollary \ref{cor kz dyn }.
\end{proof}

\begin{cor}
\label{cor V C^N}
The $V_\bla$-valued function
\bea
F_\psi(\zz;\qq;h;\kp) &=& \prod_{i=1}^N q_i^{-\frac{h\la_i}{2\kp}}
\prod_{1\le i<j\le N}\!(1-\kk_j/\kk_i)
^{\<-\frac{h\>\min(\la_i,\la_j)}{\<\kp}} \times
\\
&\times &
\prod_{1\leq i<j\leq n} \frac{\Gamma((z_i-z_j+h)/\kp)}{\Gamma((z_i-z_j-h)/\kp)} \sum_{I\in\Il} {\mc J}_{\psi,I}(-z_1/h,\dots,-z_n/h;\qq;-\kappa/h) \,v_I
\eea
is a flat section of the trigonometric dynamical connection $\pho(\nabla_{\bla,\qq,\kp})$ on $\pi_{V_\bla}$ and of the
difference connection on $\pi_{V_\bla}$ defined by the operators $\Kh_{\zz,\qq,\kp,1}\lc\Kh_{\zz,\qq,\kp,n}$ in \Ref{K_i}.

\end{cor}

\begin{proof}
The corollary follows from Lemma \ref{ttt}.
\end{proof}

\begin{cor}
\label{cor main}
For any $\psi$ as in Section \ref{HINT},
\bean
\label{stf}
&&
\\
\notag
&&
\St_{\on{id}}(F_\psi(\zz;\tilde q_1^{-1},\dots,\tilde q_N^{-1};h;-\kp)) =
\\
&&
\phantom{aaaa}
=\ \prod_{i=1}^N \tilde q_i^{-\frac{h\la_i}{2\kp}}
\prod_{1\le i<j\le N}\!(1-\tilde q_i/\tilde q_j)
^{\<\frac{h\>\min(\la_i,\la_j)}{\<\kp}} \prod_{1\leq i<j\leq n} \frac{\Gamma((z_j-z_i-h)/\kp)}{\Gamma((z_j-z_i+h)/\kp)}
\notag
\\
\notag
&&
\phantom{aaaaaaaa}
\times
\sum_{I\in\Il} {\mc J}_{\psi,I}(-z_1/h,\dots,-z_n/h;\tilde q_1^{-1},\dots,\tilde q_N^{-1};\kappa/h) \,\frac{[W_{\on{id},I}(\bs\Theta;\zz;h)]}{c_\bla(\bs\Theta)}
\eean
is a flat section of the quantum connection $\nabla_{\on{quant},\bla,\tilde\qq,\kp}$ on $H^*_T(T^*\F_\bla)$ and of the
difference connection on $H^*_T(T^*\F_\bla)$ defined by the operators
$\St_{\on{id}}\circ \Kh_{\zz,\tilde q^{-1}_1,\dots,\tilde q^{-1}_N,-\kp,1}
 \circ \nu,\dots,$ $\St_{\on{id}}\circ\Kh_{\zz,\tilde q^{-1}_1,\dots,\tilde q^{-1}_N,-\kp,n} \circ\nu$.
\end{cor}

\begin{proof}
The corollary follows from Lemma \ref{V quant}.
\end{proof}

\subsection{Example}
\label{sec exa}

Let $n=N=2$, $\bla=(1,1)$. Then $T^*\F_\bla$ is the cotangent bundle of projective line and
\bea
H^*_T(T^*\F_\bla) = \C[\ga_{1,1},\ga_{2,1}]\otimes\C[z_1,z_2]\otimes \C[h]/\langle (u-\ga_{1,1})(u-\ga_{2,1})=(u-z_1)(u-z_2) \rangle .
\eea
We have
\bea
&&
\St_{\on{id}}(F_\psi(\zz;\tilde q_1^{-1},\tilde q_2^{-1};h;-\kp))
=
\tilde q_1^{(z_1+h)/\kp}\tilde q_2^{z_1/\kp} (\tilde q_1^{-1}-\tilde q_2^{-1})^{2h/\kp}
\frac{\Gamma((z_2-z_1-h)/\kp)}{\Gamma((z_2-z_1+h)/\kp)}\times
\\
&&
\phantom{aaa}
\times
\int_\psi s^{(z_1-z_2+h)/\kp}(s-\tilde q_1^{-1})^{-h/\kp}(s-\tilde q_2^{-1})^{-h/\kp}\big[\frac{\ga_{1,1}-z_1-h}{s-\tilde q_1^{-1}} +\frac{\ga_{1,1}-z_2}{s-\tilde q_2^{-1}} \big]\,ds .
\eea
The master function is
\bea
\Phi(s;\zz;\qq;h) = s^{(z_1-z_2+h)/\kp}(s-\tilde q_1^{-1})^{-h/\kp}(s-\tilde q_2^{-1})^{-h/\kp}.
\eea
The critical point equation is
\bea
\frac{z_1-z_2+h}s - \frac h{s-\tilde q_1^{-1}} - \frac h{s-\tilde q_2^{-1}}=0
\eea
or
\bea
(z_2-z_1+h)s^2-(z_2-z_1)(\tilde q_1^{-1}+\tilde q_2^{-1})s+ \tilde q_1^{-1}\tilde q_2^{-1}(z_2-z_1-h)=0.
\eea
The discriminant of this equation is
\bean
\label{dc}
\Delta_\Phi(\zz;\tilde \qq;h)
&=&
(z_1-z_2)^2(\tilde q_1^{-1}-\tilde q_2^{-1})^2 + 4\tilde q_1^{-1}\tilde q_2^{-1}h^2 =
\\
\notag
&=&
\frac{(z_1-z_2)^2(\tilde q_1-\tilde q_2)^2 + 4\tilde q_1\tilde q_2h^2}{\tilde q_1^{2}\tilde q_2^{2}}\ .
\eean
On the other hand, the algebra of quantum multiplication $\mc H^q_\bla$ is the quotient
of $ \C[\ga_{1,1},\ga_{2,1},\zz;h]$ by the relations
\be
\gak_{1,1}+\gak_{2,1}\>=\,z_1+z_2\,,\kern3em
\gak_{1,1}\>\gak_{2,1}+\>
\frac{\tilde q_1}{\tilde q_1-\tilde q_2}\,h\>(\gak_{1,1}-\gak_{2,1}-h)\,=\,
z_1\<\>z_2\,,
\ee
see the example in \cite[Section 6]{GRTV} where again we replaced $h$ with $-h$ and $q_1,q_2$ with $\tilde q_1^{-1}, \tilde q_2^{-1}$,
respectively.
Set $\ga_{1,1}=t, \ga_{2,1}=z_1+z_2-t$. Then the relation is
\bea
(\tilde q_1-\tilde q_2) t^2 - (2\tilde q_1h+(\tilde q_1-\tilde q_2)(z_1+z_2))t
+(\tilde q_1-\tilde q_2)z_1z_2+\tilde q_1h(z_1+z_2+h)=0 .
\eea
The discriminant of this equation is
\bean
\label{dco}
\Delta_{\mc H^q_\bla}(\zz;\tilde \qq;h) =(z_1-z_2)^2(\tilde q_1-\tilde q_2)^2 + 4\tilde q_1\tilde q_2h^2.
\eean
By comparing \Ref{dc} and \Ref{dco}, we conclude that the algebra of functions on the critical set of the master function is isomorphic to
the algebra of quantum multiplication.

\end{document}